\documentclass[12pt]{amsart}
\usepackage{amssymb}
\usepackage{mathtools}
\usepackage[english,french]{babel}
\usepackage{epsfig}
\usepackage{mathptmx}
\usepackage{amsmath,amsfonts,amssymb}
\usepackage{mathtools}
\usepackage{mathrsfs}
\usepackage{color}
\usepackage{graphicx}
\usepackage{latexsym}
\usepackage{verbatim}
\usepackage{xcolor}

\newcommand{\hide}[1]{}
\numberwithin{equation}{section}

\def\eps{\varepsilon}
\def\dD{\mathop{{\rm d}_\D}}\def\domega{\mathop{{\rm d}_\Omega}}

\newcommand{\D}{\mathbb D}

\newcommand{\R}{\mathbb R}

\newcommand{\C}{\mathbb C}

\newcommand{\Aut}{{\sf Aut}(\mathbb D)}

\def\dist{{\rm dist}}

\def\Aut{{\sf Aut}}

\def\1#1{\overline{#1}}
\def\2#1{\widetilde{#1}}
\def\3#1{\widehat{#1}}
\def\4#1{\mathbb{#1}}
\def\5#1{\frak{#1}}
\def\6#1{{\mathcal{#1}}}

\newcommand{\mcite}[1]{\csname b@#1\endcsname}

\theoremstyle{theorem}

\def\dist{{\rm dist}}

\def\Aut{{\sf Aut}}

\newtheorem{theorem}{Theorem}[section]
\newtheorem{lemma}[theorem]{Lemma}

\newtheorem{corollary}[theorem]{Corollary}

\theoremstyle{definition}

\newtheorem{example}[theorem]{Example}

\theoremstyle{remark}
\newtheorem{remark}{Remark}

\newtheorem{problem1}[theorem]{Problem}

\numberwithin{equation}{section}

\title[Rigidity]{The strong form of the Ahlfors--Schwarz lemma at the\\[2mm]  boundary and a rigidity result for Liouville's equation}

\author[F. Bracci]{Filippo Bracci$^\dag$}
\address{F. Bracci: Dipartimento di Matematica, Universit\`a di Roma ``Tor Vergata", Via della Ricerca
  Scientifica 1, 00133, Roma, Italia.} \email{fbracci@mat.uniroma2.it}

\author[D. Kraus]{Daniela Kraus}
\address{D. Kraus: Department of Mathematics, University of W\"urzburg, Emil Fischer Strasse 40, 97074, W\"urzburg, Germany.} \email{dakraus@mathematik.uni-wuerzburg.de}

\author[O. Roth]{Oliver Roth}
\address{O. Roth: Department of Mathematics, University of W\"urzburg, Emil Fischer Strasse 40, 97074, W\"urzburg, Germany.} \email{roth@mathematik.uni-wuerzburg.de}


\thanks{$^\dag\,$Partially supported by  PRIN {\sl  Real and Complex Manifolds: Geometry and Holomorphic Dynamics} n. 2022AP8HZ9, by INdAM, and by the   MUR Excellence Department Project MatMod@TOV
CUP:E83C23000330006}

\long\def\REM#1{\relax}

\begin{document}
\maketitle

\selectlanguage{english}
\begin{abstract}
  We prove a boundary version of the strong form of the Ahlfors--Schwarz lemma with optimal error term. This result
  provides   nonlinear extensions of the boundary Schwarz lemma of Burns and Krantz to the class  of negatively curved conformal pseudometrics defined on arbitary hyperbolic domains in the complex plane.
  Based on a new boundary Harnack inequality for solutions of the Gauss curvature equation, we also establish a
  sharp rigidity result for conformal metrics with isolated singularities.  In the particular case of constant negative curvature this  strengthens  classical results  of Nitsche and Heins about  Liouville's equation $\Delta u=e^u$.    
\end{abstract}

\renewcommand{\thefootnote}{\fnsymbol{footnote}}
\setcounter{footnote}{2}

\section{Introduction} \label{sec:intro}

In a recent paper \cite{BKR} we established boundary versions of the strong form of the Ahlfors and Schwarz--Pick lemmas with optimal error bounds for the unit disk $\D:=\{z \in \C \, : \, |z|<1\}$. In the present paper we study  the  boundary version of the strong form of the Ahlfors--Schwarz lemma for arbitrary hyperbolic subdomains of the Riemann sphere. The emphasis is on  the sharpness of the error bound.

The motivation for extending the boundary Ahlfors--Schwarz lemma for the unit $\D$ proved in \cite[Theorem 2.6]{BKR} to other domains and investigating the sharpness of the error term is twofold.
One source of motivation comes from the fact that in the case of the unit disk the boundary Ahlfors--Schwarz lemma of \cite{BKR} provides an extension of the well--known boundary Schwarz lemma of Burns and Krantz \cite{BurnsKrantz1994}, and that much recent work (see e.g.~\cite{BaraccoZaitsevZampieri2006,LiuTang2016,TLZ17,Zimmer2018}) on boundary rigidity of holomorphic maps in one and several variables has  been devoted to relate the error term in the theorem of Burns--Krantz to the geometry of the domain.

A second motivation comes from the special case of  isolated boundary points,  a situation  which has  been of longlasting and continuing  interest in complex analysis, nonlinear PDE and two--dimensional conformal geometry, see \cite{ASVV2010,CW94,CW95,FSX2019,Gil,Heins1962,HT92,KR2008,LLX2020,McO93,Min97,Nit57,Pic1893,Pic1905,Poi1898,Yamada1988}.
 In particular, we establish with Theorem \ref{thm:3} a boundary version of the strong form of the Ahlfors--Schwarz lemma for a punctured disk with best possible remainder term. As a consequence, we obtain a rigidity result for negatively curved conformal metrics with isolated singularities in the spirit of the theorem of Burns--Krantz with an optimal error bound.
 This rigidity result seems to be new even in the case of constant negative curvature. As an application we derive   an extension of classical results of Nitsche \cite{Nit57} and Heins \cite{Heins1962} about the solutions of Liouville's equation $\Delta u=e^{u}$ with isolated singularities. The main technical tools are a new boundary Harnack--type inequality (Theorem \ref{lem:harnack}) and an accompanying Hopf lemma (Corollary \ref{cor:hopf}) for solutions of the Gauss curvature equation on punctured disks.

In the next section we  state our main results. We assume some familiarity with  basic facts about  conformal pseudometrics as described in any of the standard texts on the subject such as \cite{Marco,Ahlfors2010,BeardonMinda2007,Hayman1989,KeenLakic2007}. However, we  collect together the most relevant properties of conformal metrics  in Section \ref{sec:prelims}.  
 We prove our results in Section \ref{sec:proofs} and conclude with a list of open problems in  Section \ref{sec:questions}.

\section{Results}

Let 
$\Omega$ be a hyperbolic subdomain of the Riemann sphere $\hat{\C}=\C\cup\{\infty\}$, that is, a non--empty, open and connected subset of $\hat{\C}$ such that  $\hat{\C} \setminus \Omega$ contains at least three pairwise distinct points. We denote the hyperbolic (or Poincar\'e) metric of $\Omega$ by   $\lambda_{\Omega}(z) \, |dz|$
and the associated hyperbolic distance by $\domega$. We  emphasize that throughout this paper we use the convention that every hyperbolic metric has  constant Gauss curvature $-4$ (see Section \ref{sec:prelims}  for the precise definition of the curvature concept which we employ in this paper).

The Poincar\'e metric $\lambda_{\Omega}(z) \, |dz|$  has the following fundamental  extremal property. If $\lambda(z) \, |dz|$ is any   conformal pseudometric on $\Omega$ whose curvature $\kappa_{\lambda}$  is bounded above by $-4$, then
\begin{equation} \label{eq:ahlf1}
 \frac{\lambda(z)}{\lambda_{\Omega}(z)}\le 1 \quad \text{ for all } z \in \Omega \, .
  \end{equation}
  This is the content of Ahlfors' celebrated extension \cite{Ahlfors1938} of the Schwarz--Pick lemma.
  Heins \cite{Heins1962} (see also J{\o}rgensen \cite{Jorgensen1956}) has made the important  observation
  that if equality holds in Ahlfors' inequality (\ref{eq:ahlf1}) at just one point in the domain, then equality holds for all points.  See e.g.~Royden \cite{Royden1986}, Minda \cite{Minda1987}, Chen \cite{Chen} and also \cite[Lemma 3.2]{KR2013} for different proofs.
  This  strong form of Ahlfors' lemma  encapsulates a very useful \textit{interior} rigidity property of conformal metrics. It has the following counterpart on the boundary.

  \begin{theorem}[The strong form of the Ahlfors--Schwarz lemma at the boundary] \label{thm:main1new}
    Let $\Omega$ be a hyperbolic subdomain of $\hat{\C}$ and let $\lambda(z)\, |dz|$ be a conformal pseudometric on $\Omega$ with curvature $\le -4$. Suppose that $\{z_n\}$ is a sequence of points in $\Omega$ tending to a boundary point $p \in \partial \Omega$. If
    \begin{equation} \label{eq:anew}
  \frac{\lambda(z_n)}{\lambda_{\Omega}(z_n)}=1+o\left(
    e^{-4 \domega(z_n,q)} \right) \qquad (n \to \infty) 
  \end{equation}
  for one -- and hence for any -- $q \in \Omega$, then
  $$\lambda=\lambda_{\Omega} \, .$$ 
    \end{theorem}

    By completeness of the hyperbolic distance,   $\domega(z_n,q)
  \to +\infty$ as $z_n$ tends to the boundary point $p$. Hence Theorem~\ref{thm:main1new} says that if  $\lambda(z_n)/\lambda_{\Omega}(z_n)$ tends to $1$ sufficiently fast, that is, faster than $e^{-4 \domega (z_n,q)}$ tends to $0$, then this  forces $\lambda=\lambda_{\Omega}$.

\begin{remark}
  The error term in (\ref{eq:anew}) is given in terms of the intrinsic hyperbolic geometry of the
  domain $\Omega$.  It is well--known that for many domains (such as smoothly bounded domains or domains with corners) the error bound can also be stated in Euclidean terms if the need
  arises. 
  For instance, in the case of the unit disk $\D:=\{z \in \C \, : \, |z|<1\}$,  the asymptotic condition  (\ref{eq:anew}) takes the explicit form
\begin{equation} \label{eq:anewdisk}
   (1-|z_n|^2) \lambda(z_n) =1+o\left( (1-|z_n|)^2 \right) \,.
 \end{equation}
This special case of Theorem \ref{thm:main1new} is Theorem 2.6 in \cite{BKR}, see also \cite[Theorem 2.7.2]{Marco}. We leave the statements and proofs of such results for other domains to the reader.  
\end{remark}

\begin{remark}[The Schwarz--Pick lemma for hyperbolic derivatives of Beardon and Minda \cite{BeardonMinda2007}]
From the viewpoint of geometric function theory, the special 
 case  that  the conformal pseudometric $\lambda(z) \, |dz|$ in Theorem \ref{thm:main1new}
has the form $$\lambda(z)=\lambda_{\Omega'}(f(z)) \, |f'(z)|$$
for a holomorphic map $f$ from $\Omega$ to some other hyperbolic domain $\Omega'$
is of particular interest. It is not difficult to see -- cf.~Section \ref{sec:beardonminda} for details -- that Theorem \ref{thm:main1new} in this case 
can also be derived from  the  Schwarz--Pick lemma of Beardon and Minda for hyperbolic derivatives, see e.g.~Corollary 11.3 in \cite{BeardonMinda2007} and also \cite{Beardon1997,BeardonMinda2004}. For $\Omega=\Omega'=\D$ this has already been pointed out  in \cite[Remark 5.6]{BKR}, and the general case is analogous. We refer to   \cite{Marco,AW2009, BRW2009, EJLS2014} for some other recent results and variations on the Schwarz--Pick lemma. 
\end{remark}

\begin{remark}[Theorem \ref{thm:main1new} and the boundary Schwarz Lemma of Burns--Krantz] \label{rem:bk}
 Theorem \ref{thm:main1new}  for $\Omega=\D$ and constant curvature $-4$ contains  in particular the following result: \textit{If $f : \D \to \D$ is a holomorphic map such that
\begin{equation} \label{eq:sp23}
  (1-|z_n|^2) \frac{|f'(z_n)|}{1-|f(z_n)|^2} =1+o\left( (1-|z_n|)^2 \right)
  \end{equation}
for some sequence $\{z_n\}$ in $\D$ such that $|z_n| \to 1$, then $f$ is a conformal automorphism of $\D$}, see also \cite{BKR}.
As it has been pointed out in \cite[Remark 2.2]{BKR}, this special case of Theorem \ref{thm:main1new} generalizes the well--known boundary Schwarz lemma of Burns and Krantz \cite{BurnsKrantz1994} which asserts that the condition
\begin{equation} \label{eq:bk23}  f(z)=z+o\left(|1-z|^3\right) \qquad (z \to 1)
\end{equation}
for some holomorphic map $f : \D \to \D$ implies $f(z)=z$. 
Accordingly, already the constant curvature case of Theorem \ref{thm:main1new} provides  an extension of the boundary Schwarz lemma of Burns--Krantz  e.g.~to holomorphic maps between arbitary hyperbolic domains.
We refer to \cite{BaraccoZaitsevZampieri2006,Bol2008,Chelst2001, Dubinin2004, Dubinin2011,LiuTang2016,Osserman2000,Shoikhet2008,TLZ17,TauVla2001,Zimmer2018} for some of the many other recent extensions and generalizations of the theorem of Burns and Krantz. 
A substantial difference between Theorem \ref{thm:main1new} 
and  the boundary Schwarz lemma of \cite{BurnsKrantz1994} is the  conformally
invariant character of Theorem \ref{thm:main1new}.
\end{remark}

We now state the main results of the present paper. First, we show that under some mild topological--geometric assumptions on the domain $\Omega$ at the distinguished boundary point $p$, the error term in Theorem \ref{thm:main1new} is sharp.
Let $K$   be a connected component  of the complement $\hat{\C} \setminus \Omega$  of  a hyperbolic domain $\Omega$ in $\C$. We call $K$ \textit{nondegenerate}, if $K$ does not reduce to a single point. We call the component $K$ \textit{isolated}, if none  of its points belongs to the closure of  $\hat{\C} \setminus (\Omega \cup K)$.
Finally, a point $p \in \partial \Omega \cap \C$ is called \textit{accessible} if there is a continuous curve $\gamma : [0,1) \to \Omega$ such that $\lim_{t \to 1} \gamma(t)=p$.

\begin{theorem}[Sharpness of the strong form of Ahlfors' lemma at the boundary] \label{thm:main2new}
  Let $\Omega$ be a hyperbolic domain in $\hat{\C}$ and let $K$  be an isolated  nondegenerate  connected component
   of $\hat{\C} \setminus \Omega$. Suppose that $p  \in K\cap \partial \Omega$
   is accessible. 
  Then  there exists  a sequence $\{z_n\}$ in $\Omega$ tending to $p$
  and an  injective holomorphic  function $f : \Omega \to \D$ such that
  $$ \lim \limits_{n \to \infty} \left( \frac{\lambda_{\D}(f(z_n)) \, |f'(z_n)|}{\lambda_{\Omega}(z_n)}-1 \right) e^{4 \domega(z_n,q)} <0 \, $$
  for one -- and hence any -- $q \in \Omega$.
  \end{theorem}

  Theorem  \ref{thm:main2new}
  reveals  that   the error term in Theorem \ref{thm:main1new} is optimal  in the sense that
for any accessible boundary  point $p$ from an isolated nondegenerate  connected  component $K$ of $\hat{\C} \setminus \Omega$  one cannot replace  ``little $o$'' in (\ref{eq:anew})  by ``big $O$''. In contrast to that, 
the next result shows that if the distinguished boundary point is 
\textit{isolated}, then the error term in Theorem \ref{thm:main1new} is not   the best
possible. In fact,  the exponent $4$ in the error term (\ref{eq:anew}) can be reduced to $2$ in this case.

\begin{theorem}[The strong boundary Ahlfors lemma for hyperbolic domains with punctures] \label{thm:3}
  Let $\Omega$ be a hyperbolic domain in $\hat{\C}$  with an isolated boundary point $p$ and let
  $\lambda(z) \, |dz|$ be a conformal pseudometric on $\Omega$ with curvature $\le -4$.
    Suppose that $\{z_n\}$ is a sequence in $\Omega$ tending to $p \in \partial \Omega$. If 
  \begin{equation} \label{eq:b}
 \frac{\lambda(z_n)}{\lambda_{\Omega}(z_n)}=1+ o\left( e^{-2 \domega(z_n,q)} \right) \qquad (n \to \infty) 
 \end{equation}
 for one -- and hence for any -- $q \in \Omega$, then
$$ \lambda=\lambda_{\Omega}\, .$$
  \end{theorem}

  \begin{remark}
The improved  error term in Theorem \ref{thm:3} is optimal in the sense that  one cannot
replace ``little $o$'' in (\ref{eq:b})  by ``big $O$'', see Example \ref{exa:1} below.
\end{remark}

\begin{remark} \label{rem:euclidean}
If $p\in \C$, then
 condition (\ref{eq:b}) takes -- in Euclidean terms -- the form
$$ \frac{\lambda(z_n)}{\lambda_{\Omega}(z_n)}=1+ o\left( \frac{1}{\log \frac{1}{|z_n-p|}} \right) \, ,$$
see Lemma \ref{lem:hilfe}.
\end{remark}

\begin{remark} There is an extensive literature on the asymptotic behaviour of a conformal metric at an isolated singularity,  see \cite{ASVV2010,CW94,CW95,FSX2019,Gil,Heins1962,HT92,KR2008,LLX2020,McO93,Min97,Nit57,Pic1893,Pic1905,Poi1898,Yamada1988,Yamashita1993}. It is customary 
  and no loss of generality to consider the case that the isolated singularity is the point $p=0$  and that $\Omega$ contains the punctured unit disk $\D':=\D \setminus \{0\}$. If we  assume that the curvature $\kappa_{\lambda}$ of the metric $\lambda(z) \, |dz|$ has a continuous extension to $p=0$ with $\kappa(0)<0$, then 
  the well--known fundamental classification result \cite{Nit57,Heins1962,McO93,KR2008} asserts that 
  there is a continuous remainder function $v : \D \to \R$ such that either
  \begin{equation} \label{eq:logsing}
\log \lambda(z)=- \log |z| -\log \log (1/|z|)+ v(z) 
\end{equation}
or there is a real number $\alpha<1$ such that
\begin{equation} \label{eq:conicsing}
\log \lambda(z)=-\alpha  \log |z| + v(z) \, .
\end{equation}
The metric $\lambda(z) \, |dz|$ is said to have a \textit{logarithmic singularity} or a \textit{singularity of order $1$} if (\ref{eq:logsing}) holds and a \textit{conical singularity of order $\alpha<1$} if (\ref{eq:conicsing}) holds. The prototype for a  conformal metric with singularity of order $1$ is the Poincar\'e metric of a domain with an isolated boundary point, see \cite{Ahlfors2010,Hayman1989}.
Note that Theorem \ref{thm:3} is relevant  only for metrics with singularities of order~$1$. Conical singularities
have intensively been studied e.g.~in\cite{ASVV2010,Heins1962,TZ2003,Troyanov1991}. 
    \end{remark}

   Returning to Theorem \ref{thm:3}, we note that under the slightly more restrictive, but customary assumption that the curvature $\kappa_{\lambda}$ is locally H\"older continuous up to the isolated singularity the following  strong dichotomy arises. 
    
    \begin{theorem} \label{thm:gg}
      Let $\kappa : \D \to \R$ be a locally H\"older continuous function and 
      suppose that  $\lambda(z) \, |dz|$ is a conformal metric on $\D'$ with a singularity of order $1$ at $z=0$ and such that $\kappa_{\lambda}(z)=\kappa(z)$ for all $z \in \D'$.
      Then the following hold.
      \begin{itemize}
      \item[(a)] If 
        $\kappa(0)=-4$, then 
        \begin{equation} \label{eq:nitsche2} \log \lambda(z) =\log \lambda_{\D'}(z)+ O \left( \frac{1}{\log \frac{1}{|z|}} \right) \qquad (z \to 0) \, .
          \end{equation}
        \item[(b)] If $\kappa_{\lambda}(z) \le -4$ for all $z \in \D'$ and
          \begin{equation} \label{eq:1}
          \log \lambda(z_n) =\log \lambda_{\D'}(z_n)+o \left( \frac{1}{\log \frac{1}{|z_n|}} \right) \, .
        \end{equation}
        for a sequence $\{z_n\}$ in $\D'$ tending to $0$, then $\lambda=\lambda_{\D'}$.
        \end{itemize}
      \end{theorem}
      
Comparing (\ref{eq:1}) with  (\ref{eq:bk23}),   Theorem \ref{thm:gg} (b)  is clearly a  boundary rigidity result of Burns--Krantz type.
We  note that in fact Theorem \ref{thm:gg} (a)  can be deduced from the results \cite{KR2008}, which generalize earlier results of Nitsche \cite{Nit57} and Heins \cite{Heins1962} dealing with the constant curvature case. On the other hand, Theorem \ref{thm:gg}  (b) follows at once  from Theorem \ref{thm:3} and Remark \ref{rem:euclidean}, but  seems to be new even in the constant curvature case.
      
\begin{remark}[Constant curvature case of Theorem \ref{thm:gg}]
Clearly, the classical case  $\kappa_{\lambda}(z) =-4$ for all $z \in \D'$  in Theorem \ref{thm:gg} has been considered by many authors, including Poincar\'e \cite{Poi1898}, Picard \cite{Pic1893,Pic1905}, Heins \cite{Heins1962}, Yamada \cite{Yamada1988}, Nitsche \cite{Nit57} and Minda \cite{Min97}. In particular, Nitsche \cite{Nit57} has found the expansion

\begin{equation} \label{eq:nitsche1}
  \log  \lambda(z)= \log \lambda_{\D'}(z)+\sum \limits_{l,m,n\ge 0}a_{l,m,n} \frac{z^l \overline{z}^m}{\left( \log \frac{1}{|z|} \right)^n} \, , \qquad a_{0,0,0}=0 \, ,
\end{equation}
which obviously   implies Theorem \ref{thm:gg} (a) in the constant curvature case. In fact, (\ref{eq:nitsche1}) also provides a very weak form of Theorem \ref{thm:gg} (b) in the constant curvature case. In order to see this, one can just insert the expansion (\ref{eq:nitsche1}) into the Gauss curvature equation $\Delta \log \lambda=4 \lambda^2$. Then, as Nitsche has pointed out,  a comparison of  coefficients yields 
 $$ a_{0,0,n}=\frac{\gamma^n}{n} \qquad \text{ for all $n=1,2, \ldots$ and some constant } \gamma \in \R \, .$$
In particular, this shows that 
$$
a_{0,0,1}=0 \quad \Longrightarrow  \quad a_{0,0,n}=0 \quad \text{ for } n=0,1,2, \ldots \, .
$$ However, Theorem \ref{thm:gg} (b) implies that much more is true, namely
 $$ a_{0,0,1}=0 \quad \Longrightarrow  \quad a_{l,m,n}=0 \quad \text{ for all } l,m,n=0,1,2, \ldots \,  .$$
 This is completely  analogous to the 'analytic' form of the boundary Schwarz lemma of Burns and Krantz: if a holomorphic function $f : \D \to \D$
 has a holomorphic extension to $z=1$ of the form
 $$ f(z)=z+\sum \limits_{n=3}^{\infty} a_n (z-1)^n \, ,$$
 then
 $$ a_3=0  \Longrightarrow  \quad a_{n}=0 \quad \text{ for } n=3,4,5, \ldots \, .$$
 \end{remark}

 We finally discuss an analogue of Theorem \ref{thm:main2new} and Theorem \ref{thm:3}  for conformal metrics on $\D'$ with a conical singularity of order $\alpha<1$ at the origin. A well--known version of Ahlfors' lemma (see e.g.~\cite[p.~22]{Heins1962} or \cite[Exercise \S 3.4]{KR2013}) says that for any such pseudometric $\lambda(z) \, |dz|$ with curvature $\le -4$ the inequality
\begin{equation} \label{eq:ahlforsconic}
  \lambda(z) \le \lambda_{\alpha}(z):=(1-\alpha) \frac{|z|^{-\alpha}}{1-|z|^{2 (1-\alpha)}}  \,
  \end{equation}
 holds. Moreover,  equality occurs at some interior point of  $\D'$ if and only if $\lambda \equiv \lambda_{\alpha}$. This, of course,  is analogous to the interior case of equality in the Ahlfors--Schwarz lemma. Since  we have not been able to locate this result elsewhere, we include a proof in Section \ref{par:conic} below, see Lemma \ref{lem:conic}.
In any case, our main point here that there is even the following boundary version of this strong form of the Ahlfors--Schwarz lemma for conformal pseudometrics with conical singularities.

\begin{theorem} \label{thm:4}
  Let   $\lambda(z) \, |dz|$ be a conformal pseudometric on $\D'$
   with curvature $\le -4$ and a conical singularity of order $\alpha<1$ at $z=0$. Suppose that
  \begin{equation} \label{eq:b2}
 \frac{\lambda(z_n)}{\lambda_{\alpha}(z_n)}=1+ o\left( |z_n|^{2 (1-\alpha)} \right)
 \end{equation}
for some sequence $(z_n)$ in $\Omega$ such that $z_n \to 0$. Then
$$ \lambda=\lambda_{\alpha}\, .$$
\end{theorem}

\begin{remark}
Theorem \ref{thm:4}  sharpens  another of Nitsche's results from \cite{Nit57}, which is the following. Suppose that is  $\lambda(z) \, |dz|$ is a conformal metric on $\D'$ with a conical singularity of order $\alpha<1$ at $z=0$ and constant curvature $-4$. Then there is real number $\beta \le 1$ such that 
$$\frac{\lambda(z)}{\lambda_{\alpha}(z)}=\beta+ O\left( |z|^{2 (1-\alpha)} \right) \qquad (z \to 0) \, .$$
In particular, we see that for $\beta=1$ the same strong dichotomy arises as in the case of logarithmic singularities (see  Theorem \ref{thm:gg}). 
  \end{remark}

\section{Preliminaries} \label{sec:prelims}

In this section we briefly recall the notions and results of hyperbolic geometry we need in the paper. We refer the reader to the authorative standard texts such as \cite{Marco,Ahlfors2010,BeardonMinda2007,KeenLakic2007} for further details.

By a conformal pseudometric $\lambda(z) \, |dz|$ on a domain $G$ in $\C$ we simply mean a continuous nonnegative function $\lambda : G \to \R$. If, in addition,  $\lambda$ is  strictly positive on $G$, then $\lambda(z) \, |dz|$ is called conformal metric on $G$.  For simplicity, we assume throughout this paper that any conformal pseudometric  $\lambda(z) \, |dz|$ is of class $C^2$ on the open set $G_{\lambda}:=\{z \in G \, : \, \lambda(z)>0\}$. This allows us to define the Gauss curvature of $\lambda(z) \, |dz|$ by 
\begin{equation} \label{def:curvature}
 \kappa_{\lambda}(z):=-\frac{\Delta (\log \lambda)(z)}{\lambda(z)^2} \, , \qquad z \in G_{\lambda} \, 
\end{equation}
just using the  standard Laplacian $\Delta$.
Expressions such as ``$\kappa \le -4$'' will always mean that
$\kappa(z) \le -4$ for all $z \in G_{\lambda}$. We note that there are various ways to generalize Gauss curvature (see \cite{Ahlfors1938,Heins1962,Royden1986}), but we do not pursue these extensions in this paper.

Let $D$ and $G$ be  domains in $\C$,
$\lambda(w) \, |dw|$  a conformal pseudometric on $G$, and $f : D \to G$  a holomorphic function. Then the conformal pseudometric
$$ (f^*\lambda)(z) \, |dz|:=\lambda(f(z)) \, |f'(z)| \, |dz$$
is called the pullback of $\lambda(w) \, |dw|$ via $f$. The  fundamental property of Gauss curvature is its conformal invariance as expressed by
\begin{equation} \label{eq:curvaturedef}
  \kappa_{f^*\lambda}(z)=\kappa_{\lambda}(f(z)) \, , \qquad z \in G_{f^*\lambda} \, .
  \end{equation}
Using local coordinates and the basic invariance property (\ref{eq:curvaturedef})  one can generalize all these concepts in a straightforward way from subdomains of $\C$ to arbitrary Riemann surfaces, that is, in particular to subdomains of the Riemann sphere, see e.g.~\cite{Heins1962}.

The hyperbolic metric of the unit disk $\D$ is defined by 
 $$ \lambda_{\D}(z)\, |dz|:=\frac{|dz|}{1-|z|^2}\, ; $$
its Gauss curvature is $\kappa_{\lambda_{\D}}=-4$.
 If $G$ is a hyperbolic subdomain of  $\hat{\C}$, then uniformization theory shows that there is  a holomorphic universal covering map $\pi : \D \to G$, and one can define the hyperbolic metric $\lambda_{G}(z) \, |dz|$ by
 $$ \lambda_{G}(\pi(z)) \, |\pi'(z)| \, |dz|:=\lambda_{\D}(z) \, |dz| \, .$$
 The metric  $\lambda_{G}(z) \, |dz|$ does not depend on the choice of the covering $\pi : \D \to G$ and has constant Gauss curvature $-4$. The Schwarz--Pick lemma or the Ahlfors--Schwarz inequality (\ref{eq:ahlf1}) yields  the basic \textit{domain monotonicity principle}, which says that $\lambda_{D} \le \lambda_G$ whenever $G \subseteq D$. The hyperbolic metric $\lambda_{G}(z) \, |dz|$ induces  the hyperbolic distance $\text{d}_G$ on $G$  defined by
 $\text{d}_G(z,w):=\inf \{ \int_{\gamma} \lambda_G(w) \, |dw| \}$, 
 where the inf is taken over all piecewise smooth curves $\gamma$ in $G$ joining $z \in G$ and $w \in G$.
Then $(G,\text{d}_G)$ is a complete metric space.

\section{Proofs} \label{sec:proofs}

\subsection{Proof of Theorem \ref{thm:main1new}} \label{sec:main1new}

It is relatively straightforward to deduce Theorem \ref{thm:main1new} from its special case $\Omega=\D$, which has previously been established in \cite[Theorem 2.1]{BKR}. For completeness, we provide the details.

Fix a point $q \in \Omega$ and let
$\pi : \D \to \Omega$ be a universal holomorphic covering map with $\pi(0)=q$. For every
positive integer  $n$ there is a curve $\gamma_n$ in $\Omega$ from $q$ to $z_n$
such that
$$ \domega(z_n,q) \ge \int \limits_{\gamma_n} \lambda_{\Omega}(z) \,
|dz|-\frac{1}{n} \,.$$
Since $\pi$ has the pathlifting property there is a curve $\tilde{\gamma}_n$
in $\D$ from $0$ to some  point $w_n\in \D$ in the fibre $\pi^{-1}(\{z_n\})$ of $z_n$
such that $\pi \circ \tilde{\gamma}_n=\gamma_n$. Hence, using
\begin{equation} \label{eq:hypmettrans}
\lambda_{\Omega}(\pi(w)) \, |\pi'(w)|=\lambda_{\D}(w) \,  \qquad (w \in \D) \, ,  
\end{equation}
we see 
\begin{equation} \label{eq:we1}
 \domega(z_n,q) \ge \int \limits_{\pi \circ \tilde{\gamma}_n}
\lambda_{\Omega}(z) \, |dz|-\frac{1}{n} =\int \limits_{\tilde{\gamma}_n}
\lambda_{\D}(w) \, |dw|-\frac{1}{n} \ge 
\dD(w_n,0)-\frac{1}{n}  \, .
\end{equation}
 
On the other hand, by the distance nonincreasing property of $\pi : (\D,\dD) \to (\Omega,\domega)$, we have
\begin{equation} \label{eq:we2}
  \domega(z_n,p)=\domega(\pi(w_n),\pi(0)) \le \dD(w_n,0)\, .
\end{equation}

Since
$$\dD(w,0)=\frac{1}{2} \log
\frac{1+|w|}{1-|w|} \, , \qquad w \in \D \, , $$
the two estimates (\ref{eq:we1}) and (\ref{eq:we2}) immediately give
\begin{equation} \label{eq:estimate1}
  \lim \limits_{n \to \infty} \frac{e^{-2 \domega(z_n,q)}}{1-|w_n|}=\frac{1}{2}  \, .
  \end{equation}

  The pullback $\mu(w)\, |dw|:=\pi^*\lambda(w) \, |dw|$ of $\lambda(z) \, |dz|$ by the map $\pi: \D \to \Omega$ is a conformal pseudometric on $\D$ with curvature $\le -4$.  Using (\ref{eq:hypmettrans}), we get
  $$ \frac{\mu(w_n)}{\lambda_{\D}(w_n)}=\frac{\lambda(z_n)}{\lambda_{\Omega}(z_n)}=1+o\left(
    e^{-4 \domega(z_n,q)} \right)$$
by assumption. In view of (\ref{eq:estimate1}) this yields
  $$\frac{\mu(w_n)}{\lambda_{\D}(w_n)}=1+o \left( (1-|w_n|)^2\right)\, .$$
 Since clearly $|w_n| \to 1$, we are thus in a position to apply Theorem 2.1 in \cite{BKR}, and this shows that
  $\mu=\lambda_{\D}$ which is the same as $\lambda=\lambda_{\Omega}$.

\subsection{Theorem \ref{thm:main1new} and the Schwarz--Pick lemma of A.~Beardon and D.~Minda} \label{sec:beardonminda}

Let $\Omega, \Omega'$ be hyperbolic domains in $\C$ and $f : \Omega \to \Omega'$ a holomorphic map. Beardon \& Minda \cite{BeardonMinda2007} call the quantity
$$ f^{\Omega,\Omega'}(z):=\frac{\lambda_{\Omega'}(f(z)) \, |f'(z)|}{\lambda_{\Omega}(z)}$$
the hyperbolic distortion factor of $f$ at $z$. They have proved the following sharpening of the classical Schwarz--Pick inequality $f^{\Omega,\Omega'} \le 1$, see \cite[Corollary 11.3]{BeardonMinda2007}:
\begin{equation} \label{eq:bm}
  f^{\Omega,\Omega'}(z) \le \frac{f^{\Omega,\Omega'}(q) +\tanh \left(2 \text{d}_{\Omega}(z,q)\right)}{1+f^{\Omega,\Omega'}(q) \tanh \left(2 \text{d}_{\Omega}(z,q)\right)} \,
  \end{equation}
for all $z,q \in \Omega$. Rearranging this inequality and setting $z=z_n$  yields
$$ f^{\Omega,\Omega'}(q) \ge  \frac{ e^{4 \domega(z_n,q)} \left( f^{\Omega,\Omega'}(z_n)-1 \right)+f^{\Omega,\Omega'}(z_n)+1}{{- e^{4 \domega(z_n,q)} \left( f^{\Omega,\Omega'}(z_n)-1 \right)+f^{\Omega,\Omega'}(z_n)+1}} \, .$$
Hence, assuming
$$ \frac{\lambda_{\Omega'}(f(z_n)) \, |f'(z_n)|}{\lambda_{\Omega}(z_n)}
  =f^{\Omega,\Omega'}(z_n)=1+e \left( e^{-4 \domega(z_n,q)} \right) \qquad (n \to \infty) $$
for a sequence $\{z_n\}$ in $\Omega$ escaping to the boundary of $\Omega$   as in (\ref{eq:anew}),  shows that $f^{\Omega,\Omega'}(q)=1$. By the case of equality of the classical Schwarz--Pick lemma at the interior  point $q \in \Omega$ (see \cite[Theorem 10.5]{BeardonMinda2007}), we deduce that $f:\Omega\to \Omega'$ is a covering, so
  $$  \lambda_{\Omega'}(f(z)) \, |f'(z)| =\lambda_{\Omega}(z) \quad \text{ for all } z \in \Omega \, .$$
  We see that the special case  $\lambda=\lambda_{\Omega'}(f) \, |f'|$ for some holomorphic map $f : \Omega \to \Omega'$ in  Theorem \ref{thm:main1new} is an immediate consequence of the sharpened Schwarz--Pick inequality (\ref{eq:bm}).

\subsection{Proof of Theorem \ref{thm:main2new}}
Since the closed set $K$ is connected and contains more than one point, $G :=\hat{\C} \setminus K$ is a simply connected domain which contains $\Omega$. 
Hence, by domain monotonicity of the hyperbolic metric and hyperbolic distance,  $\lambda_{G} \le \lambda_{\Omega}$ and $\text{d}_G \le \domega$. Now fix $q \in \Omega$ and let $\psi$ be a conformal map of $G$ onto $\D$ with
$\psi(q)=0$. Since $p$ is an accessible boundary point of  $\Omega$ there is a continuous curve $\gamma : [0,1) \to \Omega$ such that $\lim_{t \to 1} \gamma(t)=p$.
It is well--known (\cite[Proposition 3.3.3]{BCDbook}) that
there exists a point $\sigma \in \partial \D$ such that $\lim_{t \to 1} \psi(\gamma(t))=\sigma$
and such that the radial limit $$\lim_{r \to 1} \psi^{-1}(r \sigma)$$ exists
and $=p$. By rotating $\psi$ if
necessary, we may  assume 
$$ p=\lim \limits_{r \to 1} \psi^{-1}(r) \, .$$
We now consider 
$$\phi(z):=z-\frac{1}{12} (z-1)^3 \, , \quad z \in \D \, . $$
Note that $\phi$ provides  the standard example for demonstrating the sharpness of the error term in the theorem of Burns--Krantz for the unit disk, see \cite{BurnsKrantz1994,EJLS2014}.
It is not difficult to show that $\phi$ is an injective holomorphic selfmap of $\D$, and a direct computation shows that for $x \in (0,1)$,
\begin{equation} \label{eq:phi}
(1-x^2) \, \left(\phi^*\lambda_{\D}\right)(x)=\left(1-x^2\right) \frac{|\phi'(x)|}{1-|\phi(x)|^2}=1-\frac{(1-x)^2}{12}+o\left( (1-x)^2 \right) \quad \text{ as } x \to 1\, .
\end{equation}
We define
$$ f:=\phi \circ \psi\, ,$$
 a (univalent) holomorphic map from  $G$ to $\D$. In view of $\lambda_{\Omega} \ge
\lambda_{G}$ and $\lambda_G=\psi^*\lambda_{\D}$, we have
\begin{equation} \label{eq:est2}
\begin{array}{rcl}
\displaystyle   \frac{\lambda_{\D}(f(w)) \, |f'(w)|}{\lambda_{\Omega}(w)} &\le & 
 \displaystyle \frac{\lambda_{\D}(f(w)) \,
   |f'(w)|}{\lambda_{G}(w)}=\frac{\lambda_{\D}(\phi(\psi(w)) \,
   |\phi'(\psi(w))|}{\lambda_{\D}(\psi(w))}\\ &=&\displaystyle \left(1-|z|^2\right) \frac{|\phi'(z)|}{1-|\phi(z)|^2}
\end{array}
\end{equation}
 for all $w=\psi^{-1}(z) \in \Omega$. 

 Since $K$ is an isolated component of $\hat{\C} \setminus \Omega$, the curve
 $x \mapsto \psi^{-1}(x)$, $x \in [0,1)$, eventually lies in $\Omega$.
In view of (\ref{eq:phi}) and (\ref{eq:est2})  we get 
 $$ \frac{\lambda_{\D}(f(w)) \, |f'(w)|}{\lambda_{\Omega}(w)}  \le
 (1-x^2) \frac{|\phi'(x)|}{1-\phi(x)^2}=1-\frac{(1-x)^2}{12}+o \left(
   (1-x)^2 \right) \qquad (x \to 1) \, .$$
In combination with $\text{d}_G \le \domega$, this leads to
 \begin{eqnarray*}
 \left(  \frac{\lambda_{\D}(f(w)) \, |f'(w)|}{\lambda_{\Omega}(w)} -1
 \right) e^{4 \domega(w,q)} &\le &   \left(  \frac{\lambda_{\D}(f(w)) \, |f'(w)|}{\lambda_{\Omega}(w)} -1
                                     \right) e^{4 \text{d}_G(w,q)}\\ &=& 
     \left(  \frac{\lambda_{\D}(f(w)) \, |f'(w)|}{\lambda_{\Omega}(w)} -1
                                                                           \right) e^{4 \dD(x,0)}\\
&=&      \left(  \frac{\lambda_{\D}(f(w)) \, |f'(w)|}{\lambda_{\Omega}(w)} -1
                                                                           \right)
    \left( \frac{1+x}{1-x} \right)^2\\
    & \le & -\frac{(1+x)^2}{12}+o(1) \qquad \text{ as } x \to 1 \, .
 \end{eqnarray*}
 Hence
 \begin{equation} \label{eq:lim1}
 \limsup \limits_{w \to p} \left(  \frac{\lambda_{\D}(f(w)) \, |f'(w)|}{\lambda_{\Omega}(w)} -1
   \right) e^{4 \domega(w,q)} \le -\frac{1}{3}  \, ,
   \end{equation}
 where this limit is taken along the curve $\psi^{-1}([0,1))$.
 
 We now provide an estimate in the opposite direction. Let $\Omega':=\psi(\Omega)$. Notice that
since $K$ is an isolated component of $\hat{\C} \setminus \Omega$  we can choose an annulus $A_r=\{z \in \C \, :
 \, r<|z|<1\}$ such that $A_{r} \subseteq \Omega' \subseteq \D$.
 Since $\psi^{-1}$ is a holomorphic map from $A_r$ to $\Omega$ the Schwarz--Pick lemma and the monotonicity property of the hyperbolic metric show that
 $$ \lambda_{\Omega}(\psi^{-1}(z)) |(\psi^{-1})'(z)| \le \lambda_{\Omega'}(z) \le \lambda_{A_r}(z) \, , \qquad z \in A_r \, .$$ 
This implies for any $z\in A_r$ and  $w:=\psi^{-1}(z) \in  \Omega$,
\begin{equation} \label{eq:est3}
 \begin{array}{rcl}
\displaystyle  \left(  \frac{\lambda_{\D}(f(w)) \, |f'(w)|}{\lambda_{\Omega}(w)} -1
   \right) e^{4 \domega(w,q)} &\ge & \displaystyle  \left(  \frac{\lambda_{\D}(f(w)) \, |f'(w)|}{\lambda_{\Omega'}(z)\, |\psi'(w)|} -1
   \right) e^{4 \domega(w,q)}\\
   &\ge &  \displaystyle\left(  \frac{\lambda_{\D}(f(w)) \, |f'(w)|}{\lambda_{A_r}(z)\, |\psi'(w)|} -1
 \right) e^{4 \domega(w,q)}\\[5mm] &=&  \displaystyle\left(  \frac{\phi^*\lambda_{\D}(z)}{\lambda_{A_r}(z)} -1
 \right) e^{4 \domega(\psi^{-1}(z),\psi^{-1}(0))} \\[5mm] &\ge & \displaystyle
\left(  \frac{\phi^*\lambda_{\D}(z)}{\lambda_{A_r}(z)} -1
 \right) e^{4 \text{d}_{A_r}(z,0)}
\end{array}
\end{equation}
where we have used again the Schwarz--Pick lemma in the last step.

We now make use of
the well--known (see e.g.~\cite{BeardonMinda2007,Sugawa1998}) explicit formula
$$ \lambda_{A_r}(z)=\frac{\pi}{2 |z| \log \left(\frac{1}{r}\right) \, \sin \left( \pi \frac{\log
      \frac{1}{|z|}}{\log \frac{1}{r}} \right)}\, ,$$
which in particular implies
$$ \text{d}_{A_r}(z,0) \sim -\frac{\log \left(\log \frac{1}{|z|}\right)}{2} \quad \text{ as } |z| \to 1 \, .$$

A tedious computation involving these expressions and the explicit formula (\ref{eq:phi}) for $\phi^*\lambda_{\D}$ yields
$$ \lim \limits_{x \to 1} \left(  \frac{\phi^*\lambda_{\D}(x)}{\lambda_{A_r}(x)} -1
 \right) e^{4 \text{d}_{A_r}(x,0)}=-\frac{\pi ^2}{6 \left(\log \left(\frac{1}{r}\right)\right)^2}-\frac{1}{3} \, .$$
In view of (\ref{eq:est3})  we hence get
 \begin{equation} \label{eq:lim1}
 \liminf \limits_{w \to p} \left(  \frac{\lambda_{\D}(f(w)) \, |f'(w)|}{\lambda_{\Omega}(w)} -1
   \right) e^{4 \domega(w,q)} \ge -\frac{1}{3}-\frac{\pi ^2}{6 \left(\log \left(\frac{1}{r}\right)\right)^2}  \, ,
   \end{equation}
   where this limit is taken along the curve $\psi^{-1}([0,1))$. In combination with the previous inequality (\ref{eq:lim1}), we see that there is a sequence $\{z_n\}$ of points on $\psi^{-1}([0,1])$ tending to $p$ such that the limit
  $$  \lim \limits_{n \to \infty} \left(  \frac{\lambda_{\D}(f(z_n)) \, |f'(z_n)|}{\lambda_{\Omega}(z_n)} -1
  \right) e^{4 \domega(z_n,q)} $$
  exists and is $<0$.
The proof of Theorem \ref{thm:main2new} is complete.
 
\subsection{Proof of Theorem \ref{thm:3}}

It is well--known that the boundary Schwarz lemma of Burns and Krantz is closely tied to a Harnack inequality for positive subharmonic functions (see \cite{BurnsKrantz1994,Ransford1995}). It is therefore to be expected 
that the key ingredient for the proof of Theorem \ref{thm:3} ought to be a Harnack--type inequality for negatively
curved conformal pseudometrics
on domains with isolated boundary points.

\begin{theorem}[Harnack inequality] \label{lem:harnack}
  Let $\Omega$ be a hyperbolic domain in $\C$ with $z=0$ as an isolated
  boundary point and let   $\lambda(z) \, |dz|$ be a  conformal pseudometric  with curvature
$\kappa_{\lambda} \le -4$ on $\Omega$.
Then for any $0<r<R:=\dist(0,\partial \Omega \setminus\{0\})$, 
\begin{equation} \label{eq:harnack}
\lambda(z) \le \left( \max \limits_{|\xi|=r}  \frac{\lambda(\xi)}{\lambda_{\Omega}(\xi)} \right)^{C_{r,R}(z)} \cdot \lambda_{\Omega}(z) \quad \text{ on } \quad 0<|z|<r \, ,
  \end{equation} 
  where
  $$C_{r,R}(z)=\frac{\log (r/R)}{\log (|z|/R)}<1\, .$$ 
\end{theorem}

In the proof of Theorem \ref{lem:harnack} we shall make use of the following auxiliary result

\begin{lemma} \label{lem:435}
  Let $\lambda(z) \, |dz|$ be a conformal pseudometric with curvature $\kappa_{\lambda} \le -4$ on a domain $\Omega \subseteq \C$.
  Then either $\lambda \equiv 0$ or $\{z \in \Omega \, : \, \lambda(z)=0\}$ is a set of capacity zero. 
\end{lemma}

\begin{proof} Since $\lambda$ is continuous and nonnegative on $\Omega$ and of class $C^2$ in $G_{\lambda}=\{z \in \Omega \, : \, \lambda(z)>0\}$ with   $\Delta \log\ge 4 \lambda^2 \ge 0$ there, it follows that
 $\log \lambda$  is subharmonic on $\Omega$ (possibly  $\log \lambda \equiv -\infty$). Hence either $\log \lambda \equiv -\infty$ or $\{z \in \Omega\, : \, \log \lambda(z)=-\infty\}$ is a set of capacity zero.
\end{proof}

\begin{proof}[Proof of Theorem \ref{lem:harnack}] It suffices to consider the case $R=1$, so $\D' \subseteq \Omega$.
  The function
  $$u(z):=\log \frac{\lambda(z)}{\lambda_{\Omega}(z)} \, $$
  is $C^2$--smooth and nonpositive on $G_{\lambda}=\{z \in \Omega \, : \, \lambda(z)>0\}$ by Ahlfors' lemma. Since $\Delta \log \lambda \ge 4 \lambda^2$ and
  $\Delta \log \lambda_{\Omega}=4 \lambda_{\Omega}^2$, the elementary
  inequality $e^{2y}-1 \ge 2 y$ which is valid for all $y \in \R$ yields
  $$ \Delta u \ge 4 \lambda_{\Omega}(z)^2 \left( e^{2 u}-1 \right) \ge 8
  \lambda_{\Omega}(z)^2 u \quad \text{ in } G_{\lambda} \, .$$
  We now define the auxiliary function
  $$ v(z):=\frac{1}{\log \frac{1}{|z|}} \, \quad z \in \D' \, , $$
  and observe by direct inspection
  $$\Delta v=8 \lambda_{\D'}(z)^2 v \, .$$
By domain monotonicity of the Poincar\'e metric, we hence get from $\D' \subseteq \Omega$ that
$$ \Delta v   \ge 8 \lambda_{\Omega}(z)^2 v  \quad \text{
    in } \, \D'\, .$$
  We next fix $r \in (0,1)$ and define
  $$ \eps:=-\frac{1}{v(r)} \max \limits_{|\xi|=r} u(\xi)  . \, $$
  Note that we can assume $\max_{|\xi|=r} u(\xi)>-\infty$, since otherwise $\log \lambda(\xi)=-\infty$ for all $|\xi|=r$  which, by  Lemma \ref{lem:435}, would force $\log \lambda  \equiv -\infty$ in which case (\ref{eq:harnack})
holds anyway. Hence $\eps \not=+\infty$. 
  Since $v$ is positive and $u \le 0$, we have $\eps \ge 0$.
We now consider  $$w:=u+\eps v \, .$$
Then, by construction,
  \begin{equation} \label{eq:bdd22}
 \Delta w \ge 8 \lambda_{\Omega}(z)^2 w \, \quad \text{ in } \quad G_\lambda
 \cap \D'  \, ,
 \end{equation}
  and
  \begin{equation} \label{eq:bdd23}
    w(z) \le 0 \quad \text{ for all } |z|=r \, .
    \end{equation}
 Since
  $$ \lim \limits_{z \to 0} v(z)=0 \, , $$
  we also have
  \begin{equation} \label{eq:bdd24}
    \limsup \limits_{z \to 0} w(z) \le 0 \, .
    \end{equation}
We now show that this implies
  \begin{equation} \label{eq:55}
w(z) \le 0 \quad \text{ for all } 0<|z| \le r \, .
\end{equation}
To prove this, we assume that $w$ is positive at some point inside $0<|z| \le
r$. In view of the boundary conditions (\ref{eq:bdd23}) and (\ref{eq:bdd24}),
the function  $w$ then attains its positive maximum on $0<|z| \le r$  at some point $z_0$
with $0<|z_0|<r$. In particular, $z_0 \in G_{\lambda}$ and hence  $w$ is of class $C^2$ in
a neighborhood of $z_0$,  so $\Delta w (z_0) \le 0$ by elementary calculus. In
view of (\ref{eq:bdd22}), this however yields
$$ 0 \ge \Delta w(z_0) \ge 8 \lambda_{\Omega}(z_0)^2 w(z_0)>0 \, ,$$
a contradiction. Hence, we have proved (\ref{eq:55}), which is equivalent to (\ref{eq:harnack}).
\end{proof}

\begin{corollary}[Hopf] \label{cor:hopf}
 Let $\Omega$ be a hyperbolic domain in $\C$ with $z=0$ as an isolated
  boundary point and let   $\lambda(z) \, |dz|$ be a  conformal pseudometric  with curvature
$\kappa_{\lambda} \le -4$ on $\Omega$. Then either $\lambda \equiv
\lambda_{\Omega}$ or
  \begin{equation} \label{eq:hopf}
\limsup \limits_{z \to 0} \left( \log \frac{\lambda(z)}{\lambda_{\Omega}(z)}\right)  \cdot \log \frac{1}{|z|}<0 \, .
\end{equation}
  \end{corollary}

\begin{proof}
  As we have already noted, the well--known  case of equality at some
  interior point in the
Ahlfors--Schwarz lemma shows that $\lambda(z)<\lambda_{\Omega}(z)$ for all $z \in
\Omega$ provided that $\lambda(z_0)<\lambda_{\Omega}(z_0)$ for one $z_0 \in
\Omega$. Hence
(\ref{eq:hopf}) follows immediately from (\ref{eq:harnack}). 
\end{proof}

\begin{remark} An inspection of the argument above shows $$ \limsup \limits_{z \to 0} \log
  \frac{\lambda(z)}{\lambda_{\Omega}(z)}\cdot \log \frac{1}{|z|} \le \inf
\limits_{0 \le r \le R} \left(\max
\limits_{|\xi|=r} \log \frac{\lambda(\xi)}{\lambda_{\Omega}(\xi)} \right)
\cdot \log
\frac{R}{r} \, 
$$
where $R=\dist(0,\partial \Omega \setminus \{0\})$.
  \end{remark}

\begin{remark} \label{rem:123}
We comment briefly on the choice of the auxiliary function
$$v(z)=\frac{1}{\log \frac{1}{|z|}}$$
in the proof of Theorem \ref{lem:harnack}. Using polar
coordinates,  it is straightforward to verify that  the set  of  all  \textit{radially
  symmetric solutions} of the linear PDE
$$ \Delta v=8 \lambda_{\D'}(z)^2 v$$
is a vector space which is generated by the two linearly independent solutions
\begin{equation} \label{eq:rad}
 \frac{1}{\log \frac{1}{|z|}} \, , \qquad  \left( \log \frac{1}{|z|} \right)^2
 \, .
\end{equation}
Among these solutions the auxiliary function is the one which stays bounded as $z \to 0$.
\end{remark}

For the proof of Theorem \ref{thm:3}  we need another lemma.

\begin{lemma} \label{lem:hilfe}
  Let $\Omega \subseteq \C$ be a hyperbolic domain with an isolated boundary point  at $p=0$.
  Then for any $0<R<\dist(0, \partial \Omega \setminus \{0\})$ and any $q \in \Omega$ with $0<|q|<R$ there are constants $c_1,c_2>0$ such that
  $$ c_1 \le \left( \log  \frac{1}{|z|} \right) e^{-2 \text{d}_{\Omega}(z,q)}  \le c_2 \quad \text{ for all } 0<|z| \le R \, .$$ 
\end{lemma}

\begin{proof}
  This is essentially well--known. We can assume $\D' \subseteq \Omega \subseteq \C'':=\C \setminus \{0,1\}$, so
  $\text{d}_{\C''} (z,q) \le \domega(z,q) \le \text{d}_{\D'}(z,q)$ for all $z,q \in \D'$. Then 
  \begin{eqnarray*}
   \domega(z,q) & \le&   \text{d}_{\D'}(z,q)  \le   \text{d}_{\D'}\left(z,z \frac{|q|}{|z|}\right)+\gamma \quad \text{ with } \gamma:=\max \limits_{|w|=|q|} \text{d}_{\D'}(w,q) \\
                        &=&\frac{1}{2} \left| \log \frac{\log |z|}{\log |q|} \right|+\gamma \, ,
 \end{eqnarray*}
 as can be seen by a direct computation using the explicit expression
 $$ \lambda_{\D'}(z)=\frac{1}{2 |z| \log \frac{1}{|z|}} \, .$$
 On the other hand, we have
 $$ \domega(z,q) \ge \text{d}_{\C''}(z,q) >\frac{1}{2} \log  \frac{|\log |z||}{\pi+|\log |q||} \, ,$$
see \cite[(9.4.21)]{Hayman1989}. The claim follows if we set 
 $$ c_1:=|\log |q|| \cdot e^{-2 \gamma} \,  , \qquad c_2:=|\log |q||+\pi \, .$$
  \end{proof}

We can now give the 
\begin{proof}[Proof of Theorem \ref{thm:3}]
  By Lemma \ref{lem:hilfe},  condition (\ref{eq:b}) is equivalent to 
  $$ \lim \limits_{n \to \infty}  \left( \log \frac{\lambda(z_n)}{\lambda_{\Omega}(z_n)}\right)  \cdot \log \frac{1}{|z_n|}=0 \, . $$ 
Hence  Corollary \ref{cor:hopf} forces $\lambda=\lambda_{\Omega}$.
\end{proof}

\subsection{Sharpness of Theorem \ref{thm:3}}

The following example illustrates that the error term in Theorem \ref{thm:3} is optimal.

\begin{example} \label{exa:1}
Let $f : \D' \to \D'$ be defined by
$$f(z):=z \exp \left(-\frac{1+z}{1-z} \right) \, .$$
Then
\begin{equation} \label{eq:sharp1}
 \lim\limits_{z \to 0} \left( \frac{\lambda_{\D'}(f(z)) \,
     |f'(z)|}{\lambda_{\D'}(z)}-1 \right) \log \frac{1}{|z|} =-1 \, .
 \end{equation}
Hence  we see that we cannot replace ``little o'' by ``big O'' in Theorem \ref{thm:3}.
In order to prove (\ref{eq:sharp1}), we just note that
$$ f'(z)=\frac{1-4z+z^2}{(1-z)^2} \frac{f(z)}{z} \, ,$$
and hence
$$ \frac{\lambda_{\D'}(f(z)) \,
  |f'(z)|}{\lambda_{\D'}(z)}=\frac{|1-4z+z^2|}{|1-z|^2} \frac{\log
  \frac{1}{|z|}}{\log \frac{1}{|f(z)|}}= \frac{|1-4z+z^2|}{|1-z|^2} \frac{\log
  \frac{1}{|z|}}{\log \frac{1}{|z|}+\frac{1-|z|^2}{|1-z|^2}}\, .$$
This implies 
$$\left(  \frac{\lambda_{\D'}(f(z)) \,
  |f'(z)|}{\lambda_{\D'}(z)}-1 \right) \log \frac{1}{|z|}=\left( \frac{|1-4z+z^2|}{|1-z|^2}
\frac{1}{1+\frac{1}{\log \frac{1}{|z|} \frac{1-|z|^2}{|1-z|^2}}}-1 \right)
\log \frac{1}{|z|}\to -1 
$$
as $z \to 0$.
\end{example}

\subsection{Proof of Theorem \ref{thm:gg}} \label{par:gg}

We first prove part (a). By Theorem 3.5 in \cite{KR2008}, we have
$$ \log \lambda(z)=-\log |z|-\log \log (1/|z|)+w(z) \, , \qquad z\in \D', $$
with a continuous function $v : \D \to \R$ satisfying $v(0)=-\log \sqrt{-\kappa(0)}=-\log 2$, and the same conclusion holds for $\log \lambda_{\D'}$. Hence
$$ \log  \lambda(z)=\log \lambda_{\D'}(z)+w(z) \, , \qquad z\in \D',$$
with a continuous functions $w : \D \to \R$ satisfying $w(0)=0$. Now, under the assumption that $\kappa$ is locally H\"older continuous  in a neighborhood of $z=0$, Theorem 1.1 in \cite{KR2008} controls the growth of  the gradient $\nabla w$ of $w$ as follows
$$ \nabla w(z)=O \left( \frac{1}{|z| \left( \log \frac{1}{|z|} \right)^2} \right) \qquad (z \to 0) \, .$$
In combination with $w(0)=0$ and
$$ \int \frac{dx}{x \left( \log \frac{1}{x} \right)^2}=\frac{1}{\log \frac{1}{x}} \, , $$
an elementary integration therefore gives
$$ w(z)= O \left( \frac{1}{\log \frac{1}{|z|} } \right) \qquad (z \to 0) \, ,$$
which proves part (a).

To prove (b) we just note that $ \log (1+x)=x +o(x)$  as $x \to 0$. In view of Lemma \ref{lem:hilfe},
 condition (\ref{eq:1}) therefore implies   condition (\ref{eq:b}) of  Theorem \ref{thm:3}, and part (b) follows from Theorem \ref{thm:3}.

\subsection{Proof of Theorem \ref{thm:4}} \label{par:conic}

The proof of Theorem \ref{thm:4} is more or less identical  to the proof of Theorem \ref{thm:3}, so we only indicate the main steps. The only real difference is that the 'working horse' is now the auxiliary function
$$ v_{\alpha}(z):=\frac{|z|^{2(1-\alpha)}}{1-|z|^{2(1-\alpha)}} \, .$$
However, unlike in the case of a logarithmic singularity (see Remark \ref{rem:123}), we do not have any convincing explanation for this particular choice. Frankly speaking, we have found $v_{\alpha}$ by  educated guessing.

We begin, however, with an analogue  of the interior case of the strong  Ahlfors--Schwarz lemma for the situation of conformal pseudometrics with conical singularities.

\begin{lemma} \label{lem:conic}
Let   $\lambda(z) \, |dz|$ be a regular conformal pseudometric  with curvature
$\kappa_{\lambda} \le -4$ on $\D'$ and a conical singularity at $z=0$ of order $\alpha \in (0,1)$.
Suppose that
$$ \lambda(z_0)=\lambda_{\alpha}(z_0) \quad \text{ for some } z_0 \in \D' \, . $$
Then
$$ \lambda(z)=\lambda_{\alpha}(z) \quad \text{ for all } z \in \D' \, .$$
  \end{lemma}

\noindent
  The proof is a standard application of the standard Hopf maximum principle, see~\cite[Thm.~3.5]{GT97}; we include it for completeness.

 \begin{proof}
    We  show that any point $z_0 \in \D'$ has an open neighborhood such that either $\lambda=\lambda_{\alpha}$ in this neighborhood or $\lambda(z)<\lambda_{\alpha}(z)$ in this neighborhood. Since $\D'$ is connected, the conclusion of the lemma follows.     Fix $z_0 \in \D'$. If $\lambda(z_0)<\lambda_{\alpha}(z_0)$, then $\lambda(z)<\lambda_{\alpha}(z)$ for all $z$ in some open neighborhood of $z_0$ just by continuity of $\lambda$ and $\lambda_{\alpha}$. Now let $\lambda(z_0)=\lambda_{\alpha}(z_0)>0$, so $z_0 \in G_{\lambda}=\{z \in \D' \, : \, \lambda(z)>0\}$.
 The function
  $$u(z):=\log \frac{\lambda(z)}{\lambda_{\alpha}(z)} \, $$
  is $C^2$--smooth  on $G_{\lambda}$.
  Since $\Delta \log \lambda \ge 4 \lambda^2$ and
  $\Delta \log \lambda_{\alpha}=4 \lambda_{\alpha}^2$, the elementary
  inequality $e^{2y}-1 \ge 2 y$ which is valid for all $y \in \R$ yields
  $$ \Delta u \ge 4 \lambda_{\alpha}(z)^2 \left( e^{2 u}-1 \right) \ge
  8
  \lambda_{\alpha}(z)^2 u \quad \text{ in } G_{\lambda} \, .$$
  Let $D$ denote the connected component of the open set $G_{\lambda}$  which contains the point $z_0$. Since $u(z_0)=0$ and  $u(z) \le 0$ for all $z \in \D'$ by (\ref{eq:ahlforsconic}), the function $u$ attains its maximum on $D$ at $z_0$. In view of the fact that $\lambda_{\alpha}$ is locally bounded on $D$, we are in a position to 
apply the  Hopf maximum principle. We  conclude that  $u$ has to be constant
in $D$, so $\lambda(z)=\lambda_{\alpha}(z)$ for all $z \in D$.
    \end{proof}
  
\begin{theorem}[Harnack inequality for conical singularities] \label{thm:harnackconic}
Let   $\lambda(z) \, |dz|$ be a regular conformal pseudometric  with curvature
  $\kappa_{\lambda} \le -4$ on $\D'$ and a conical singularity at $z=0$ of order $\alpha \in (0,1)$.
  Suppose that
  $$ \lim \limits_{z \to 0} \frac{\lambda(z)}{\lambda_{\alpha}(z)}=1 \, .$$
Then for any $0<r<1$,
\begin{equation} \label{eq:harnacknew}
\lambda(z) \le \left( \max \limits_{|\xi|=r}  \frac{\lambda(\xi)}{\lambda_{\alpha}(\xi)} \right)^{\gamma_{\alpha,r}(z)} \cdot \lambda_{\alpha}(z) \quad \text{ on } \quad 0<|z|<r \, ,
  \end{equation} 
  where
  $$\gamma_{\alpha,r}(z)=\frac{v_{\alpha}(z)}{v_{\alpha}(r)} <1\, .$$ 
\end{theorem}

\begin{proof} In the proof of Lemma \ref{lem:conic}, we have already observed that 
  the function
  $$u(z):=\log \frac{\lambda(z)}{\lambda_{\alpha}(z)} \, $$
  is continuous on $\D$ and $C^2$--smooth and nonpositive on $G_{\lambda}$ with 
  $$ \Delta u \ge 
  8
  \lambda_{\alpha}(z)^2 u \quad \text{ in } G_{\lambda} \, .$$
 Observe that
  $$\Delta v_{\alpha}\ge8 \lambda_{\alpha}(z)^2 v_{\alpha} \quad \text{ in } \D' \supseteq G_{\lambda} \, .$$
  We next fix $r \in (0,1)$ and define
  $$ \eps:=-\frac{1}{v_{\alpha}(r)} \max \limits_{|\xi|=r} u(\xi) . \, $$
As before we we can assume $\max_{|\xi|=r} u(\xi)>-\infty$, so $\eps \in [0,+\infty)$.
Then, from the  construction and using
$$ \lim \limits_{z \to 0} v_{\alpha}(z)=0 \, , $$
we see that the function $w:=u+\eps v_{\alpha}$ fulflills
  \begin{equation} \label{eq:bdd22new}
 \Delta w \ge 8 \lambda_{\alpha}(z)^2 w \, \quad \text{ in } \quad G_\lambda
 \cap \D'  \, ,
 \end{equation}
  and
  \begin{equation} \label{eq:bdd23new}
    w(z) \le 0 \quad \text{ for all } |z|=r \, .
    \end{equation}
as well as 
  \begin{equation} \label{eq:bdd24new}
    \limsup \limits_{z \to 0} w(z) \le 0 \, .
    \end{equation}
  Making use of (\ref{eq:bdd22new})--(\ref{eq:bdd24new}), we argue as before to conclude that
  \begin{equation} \label{eq:55new}
w(z) \le 0 \quad \text{ for all } 0<|z| \le r \, .
\end{equation}
This is equivalent to 
(\ref{eq:harnacknew}).
\end{proof}

\begin{corollary}[Hopf lemma for conical singularities] \label{cor:hopfconic}
Let   $\lambda(z) \, |dz|$ be a conformal pseudometric  with curvature
$\kappa_{\lambda} \le -4$ on $\D'$ and a conical singularity at $z=0$ of order $\alpha <1$. Then either $\lambda \equiv
\lambda_{\alpha}$ or
  \begin{equation} \label{eq:hopfnew}
\limsup \limits_{z \to 0} \left(\log \frac{\lambda(z)}{\lambda_{\alpha}(z)} \right) \cdot  |z|^{2 (\alpha-1)}<0 \, .
\end{equation}
  \end{corollary}

  \begin{proof} We note that $\lambda/\lambda_{\alpha}$ has a continuous extension to $z=0$, so we may assume that 
    $$ \lim \limits_{z \to 0} \frac{\lambda(z)}{\lambda_{\alpha}(z)}=1 \, $$
    since otherwise this limit is $<1$ and (\ref{eq:hopfnew}) is trivially true. 
      As we have already noted, we have   $\lambda(z)<\lambda_{\alpha}(z)$ for all $z \in
\D'$ provided that $\lambda(z_0)<\lambda_{\alpha}(z_0)$ for one $z_0 \in
\D'$. Hence
(\ref{eq:hopf}) follows immediately from (\ref{eq:harnacknew}). 
\end{proof}

\begin{remark} An inspection of the argument above shows $$ \limsup \limits_{z \to 0} \left( \log
  \frac{\lambda(z)}{\lambda_{\alpha}(z)} \right)\cdot  |z|^{2 (\alpha-1)} \le \inf
\limits_{0 \le r < 1} \left[ \left(\max
\limits_{|\xi|=r} \log \frac{\lambda(\xi)}{\lambda_{\alpha}(\xi)} \right)
\cdot r^{2(\alpha-1)} \right]   \, .
$$
  \end{remark}

\section{Some questions} \label{sec:questions}

\subsection{}  Theorem \ref{thm:3} and Theorem \ref{thm:main2new} cover by far not all possible situations. It would be interesting to investigate the sharpness of the error bound in the boundary Ahlfors--Schwarz lemma (Theorem \ref{thm:main1new}) when neither of the hypotheses  of Theorem \ref{thm:3} nor Theorem \ref{thm:main2new} hold.

\subsection{} As observed in Remark \ref{rem:bk} the boundary Ahlfors--Schwarz lemma for the case of the unit disk (\cite[Theorem 2.6]{BKR}) provides an extension of 
the boundary Schwarz lemma of Burns and Krantz.
Baracco, Zaitsev and Zampieri \cite{BaraccoZaitsevZampieri2006} have proved the following strengthening of the theorem of  Burns and Krantz: If $f : \D \to \D$ is holomorphic and if
$$ f(z_n)=z_n+o(|p-z_n|^3)$$
for some sequence $\{z_n\}$ in $\D$ converging to some point $p \in\partial \D$, then $f(z) \equiv z$.
\begin{problem1}
Does this strengthened version of the Burns--Krantz theorem also follow from the boundary Ahlfors--Schwarz lemma for the unit disk (Theorem \ref{thm:main1new} for $\Omega=\D$)? 
\end{problem1}

\subsection{}
Let $f : \D \to \D$ be holomorphic and $p \in \partial \D$. We have seen in Remark \ref{rem:bk}
that  if
$$
\left(1-|z|^2 \right) \frac{|f'(z)|}{1-|f(z)|^2} =1+o( (1-|z|)^2) \qquad \text{ as } z \to p \,  $$
along just one sequence $\{z_n\}$ tending to $p$, then $f \in \Aut(\D)$. In particular,  $f$ has a holomorphic extension to $z=p$.
In an earlier paper \cite{KRR2007} we have shown that if
\begin{equation} \label{eq:BddAhlfors}
  \left(1-|z|^2 \right) \frac{|f'(z)|}{1-|f(z)|^2} =1+o(1)
\end{equation}
as $z \to p$ \textit{unrestricted} in $\D$ or even only
$$ \liminf\limits_{z \to p} \frac{|f'(z)|}{1-|f(z)|^2}=+\infty \, $$
(again unrestricted) then $f$ still has a  holomorphic extension to $z=p$. However, it was shown in \cite{Zorboska2015}  that if one merely assumes that 
\begin{equation} \label{eq:nina}
  \angle \lim \limits_{z \to p} \left(1-|z|^2 \right) \frac{|f'(z)|}{1-|f(z)|^2}=1
\end{equation}
(\textit{angular limit}), then $f : \D  \to \D$ does not need to have an angular derivative at $z=p$.
\begin{problem1}
Does  (\ref{eq:nina}) imply that $f$ has an angular limit at $z=p$\,?
\end{problem1}

\begin{problem1}
Suppose that $w$ is continuous at $z=p$ with $w(0)=0$ and
$$ \left(1-|z|^2 \right) \frac{|f'(z)|}{1-|f(z)|^2}=1+o(w(z)) \, $$
as $z \to p$ \textit{angularly}. Under which conditions on $w$  does $f$ have an angular derivative at $z=p$? 
\end{problem1}

We finally note that the boundary behaviour of the quantity
$$  \frac{|f'(z)|}{1-|f(z)|^2}$$
on the \textit{entire} unit circle in connection whether $f$ is inner or not has been investigated in \cite{AAN,Kraus2013}, see Theorem 3 in \cite{AAN} and Lemma 2.10 in \cite{Kraus2013}.

\end{document}